\documentclass[12pt]{article}
\usepackage{amssymb,amsmath,amsbsy, amsthm}
\usepackage[english]{babel}
\usepackage{amscd}
\usepackage{pb-diagram}
\usepackage{pstricks,pst-node}
\usepackage{amsfonts}
\usepackage{latexsym}
\usepackage{pstricks,pst-node}
\usepackage{graphicx}
\usepackage{pstricks,pst-node}
\usepackage{amsfonts}
\usepackage{latexsym}
\usepackage{pstricks,pst-node}
\usepackage{tikz}
\usepackage{float}
\usepackage{enumerate}

\newtheorem{theorem}{Theorem}[section]
\newtheorem{lemma}[theorem]{Lemma}

\newtheorem{proposition}[theorem]{Proposition}

\newtheorem{definition}[theorem]{Definition}

\numberwithin{equation}{section}


\newcommand{\I}{\mathbb{I}}

\newcommand{\F}{\mathcal{F}}
\newcommand{\f}{\widehat{f}}
\newcommand{\K}{\mathcal{K}}
\newcommand{\U}{\mathcal{U}}

\newcommand{\C}{\mathcal{C}}

\newcommand{\sub}{\subseteq}

\begin{document}
\date{}
\title{Transitivity of some uniformities on fuzzy sets}        
\author{Daniel Jard\'on\footnote{ORCID: 0000-0001-9054-0788}, Iv\'an S\'anchez\footnote{ORCID: 0000-0002-4356-9147} \footnote{Corresponding author: Iv\'an S\'anchez}\, and Manuel Sanchis\footnote{ORCID: 0000-0002-5496-0977}}

\maketitle

\begin{abstract}
Given a uniform space $(X, \U)$, we denote by $\mathcal{F}(X)$ to the family of all normal upper semicontinuous fuzzy sets $u \colon X \to [0,1]$ with compact support. In this paper, we study transitivity on some uniformities on $\mathcal{F}(X)$: the level-wise uniformity $\U_{\infty}$, the Skorokhod uniformity $\U_{0}$, and the sendograph uniformity $\U_S$. If $f \colon (X, \U) \to (X, \U)$ is a continuous function, we mainly characterize when the induced dynamical systems $\f \colon (\F(X), \U_{\infty}) \to (\F(X), \U_{\infty})$, $\f \colon (\F(X), \U_{0}) \to (\F(X), \U_{0})$ and $\f \colon (\F(X), \U_{S}) \to (\F(X), \U_{S})$ are transitive, where $\f$ is the Zadeh's extension of $f$. 

\end{abstract}

\noindent {\bf Keywords:} Fuzzy set; Level-wise uniformity; Skorokhod uniformity; Sendograph uniformity; Zadeh's extension; Transitivity; Weakly mixing.

\section{Introduction}

A \textit{dynamical system} is a pair $(X,f)$, where $X$ is a topological space and $f \colon X \to X$ is a continuous function.  Let us recall that a dynamical system $(X,f)$ is \textit{transitive} if for each non-empty open subsets $U$ and $V$ of $X$, there exists $n \in \mathbb{N}$ such that $f^{n}(U)\cap V \neq \emptyset$. For instance, consider $S^{1}=\{x\in \mathbb{C}: ||x||=1\}$ with its usual metric, and $\theta \in \mathbb{R}$. Define $f_{\theta}:S^{1} \to S^{1}$ by $f_{\theta}(e^{i\lambda})=e^{i(\lambda+2\pi\theta)}$. Then $f_{\theta}$ is transitive if and only if $\theta$ is irrational (see \cite{RF}). It is also known that the tent function:
\begin{align*}
f(x)= \left\{ \begin{array}{lcc}
             2x, & \text{if  }  \hspace{0.3cm} 0\leq x \leq \frac{1}{2}. \\
             2(1-x), &  \text{if } \hspace{0.3cm} \frac{1}{2}\leq x \leq 1.
          \end{array}
   \right.\end{align*}
\noindent is transitive on $[0,1]$ with its usual metric.

Given a space $X$, we denote by $\mathcal{K}(X)$ to the family of non-empty compact subsets of $X$. If $(X,d)$ is a metric space, then $\mathcal{K}(X)$ is endowed with the Hausdorff metric $d_H$. If $f: X \to X$ is continuous, then we define the extension $\overline{f}: \mathcal{K}(X) \to \mathcal{K}(X)$ by $\overline{f}(A)=f(A)$, and it is also continuous on $(\mathcal{K}(X), d_H)$ \cite{Michael}. So, we can consider the dynamical system $(\mathcal{K}(X), \overline{f})$. In \cite{RF}, the author proved the following: If $f: X \to X$ is continuous and $\overline{f}$ is transitive, then $f$ is transitive. On the other hand, if $\theta$ is irrational, then $f_{\theta}$ is transitive on $S^{1}$, but $\overline{f_{\theta}}$ is not transitive. It is also shown that if $X=[0,1]$ with its usual metric and $f: X \to X$ is the tent function, then $\overline{f}$ is transitive on $X$ \cite{RF}. Hence, the latter facts suggest the following question: Under which conditions is $\overline{f}$ transitive?

The question above was answered independently by Banks and Peris in the realm of topological spaces (see \cite{Banks, Peris}). We need some notations in order to present the result of Banks and Peris. Let $(X, \tau)$ be a topological space. Then $\mathcal{K}(X)$ will be endowed with the Vietoris topology $\tau_{V}$. Given a continuous function $f:(X, \tau) \to (X, \tau)$, as in metric spaces, we can consider the induced function $\overline{f}: (\mathcal{K}(X), \tau_V) \to (\mathcal{K}(X), \tau_V)$. If $(X, d)$ is a metric space and $\tau_d$ is the topology on $X$ induced by the metric $d$, then the Vietoris topology $(\tau_d)_{V}$ on $\mathcal{K}(X)$ coincides with the topology $\tau_{d_H}$. Given a topological space $X$, a continuous function $f: X \to X$ is \textit{weakly mixing} if $f\times f$ is transitive on $X\times X$. We are now ready to introduce the result of Banks and Peris: Let $f: X \to X$ be a continuous function on a topological space $X$. Then $f$ is weakly mixing if and only if $\overline{f}$ is transitive \cite{Banks, Peris}.

Given a topological space $X$, we denote by $\mathcal{F}(X)$ to the family of normal upper semicontinuous fuzzy sets on $X$ with compact support. If $(X, d)$ is a metric space, then the hyperspace $\mathcal{F}(X)$ can be equipped with different metrics: the level-wise metric $d_\infty$, the Skorokhod metric $d_0$, the sendograph metric $d_S$, among others.

Let $X$ be a topological space and $f: X \to X$ a continuous function, then the \textit{Zadeh's extension} of $f$ to $\mathcal{F}(X)$ is denoted by $\widehat{f}$ and defined as follows:
\begin{align*}
\widehat{f}(u)(x)= \left\{ \begin{array}{lcc}
             \sup\{u(z): z\in f^{-1}(x)\}, & \text{if  }  \hspace{0.3cm} f^{-1}(x)\neq \emptyset. \\
             0, &  \text{if } \hspace{0.3cm}  f^{-1}(x)= \emptyset.
          \end{array}
   \right.\end{align*}

\noindent Notice that $\widehat{f}: \mathcal{F}(X) \to \mathcal{F}(X)$. The result of Banks and Peris was extended to $\mathcal{F}(X)$ as follows: If $(X,d)$ be a metric space, then $((X,d),f)$ is weakly mixing if and only if $((\mathcal{F}(X), d_\infty), \widehat{f})$ is transitive; if and only if $((\mathcal{F}(X), d_0), \widehat{f})$ is transitive; if and only if $((\mathcal{F}(X), d_S), \widehat{f})$ is transitive (see \cite{Transitivity}).

It is a well-known fact that every uniformity $\U$ on a set $X$ induces a topology $\tau_\U$ on $X$. To be precise, the topology $\tau_\U$ is the family $\{V\sub X\, :\, \text{for every } x\in V \text{, there exists } U \in \U \text{ such that } U(x)\sub V\}$, where  $U(x)=\{y\in X\, : \, (x,y)\in U\}$. In this case, the topological space $(X,\tau_\U)$ is a Tychonoff space (for the details, we refer to the reader to Chapter~8 of the classic text  \cite{Engelking}).  If $(X,\U)$ is a uniform space, then $\F(X)$ is the family of all normal upper semicontinuous fuzzy sets on $(X,\tau_\U)$ with compact support. In \cite{JSS23}, we introduce some uniformities on the set $\F(X)$: the level-wise uniformity $\U_{\infty}$, the Skorokhod uniformity $\U_{0}$, and the sendograph uniformity $\U_S$. In this paper, we show that if $f: (X, \U) \to (X, \U)$ is continuous, then $\f: (\F(X), \U_{\infty}) \to (\F(X), \U_{\infty})$, $\f: (\F(X), \U_{0}) \to (\F(X), \U_{0})$, and $\f: (\F(X), \U_{S}) \to (\F(X), \U_{S})$ are continuous, where $\f$ is the Zadeh's extension of $f$ (see Theorems \ref{Th:Conti-U_infinito}, \ref{Th:Conti-U_0}, and \ref{Th:Conti-U_S}). We mainly prove that if $f \colon (X, \U) \to (X, \U)$ is a continuous function, then $f: (X, \U) \to (X, \U)$ is weakly mixing if and only if
$\f \colon (\F(X), \U_{\infty}) \to (\F(X), \U_{\infty})$ is transitive; if and only if $\f \colon (\F(X), \U_{0}) \to (\F(X), \U_{0})$ is transitive; if and only if $\f \colon (\F(X), \U_{S}) \to (\F(X), \U_{S})$ is transitive (see Theorem \ref{Th:transitivity}).

\section{Preliminaries}

In this section, we introduce the results on fuzzy theory that we need in the sequel. A fuzzy set $u$ on a topological space $X$ is a function $u \colon X \to 	\I	$, where $	\I	$ denotes the closed unit interval $[0,1]$. For each $\alpha 	\in  (0,1]$, we define the $\alpha$-\textit{cut} of $u$ as the set $u_\alpha=\{x	\in X: u(x)\geq \alpha\}$ . The \textit{support} of $u$, denoted by $u_0$, is the set $\overline{\{x \in X:u(x)>0\}}$. Let us note that $u_0=\overline{\bigcup\{u_\alpha:{\alpha	\in (0,1]}\}}$. Let $\mathcal{F}(X)$ be the family of all normal upper semicontinuous fuzzy sets $u \colon X \to \mathbb{I}$ with compact support. 

Let $X$ be a topological space. If $f \colon X \to X$ is a function, the Zadeh's extension of $f$ to $\F(X)$ is denoted by $\f \colon \F(X) \to \F(X)$ and defined as follows:

$$\f(u)(x)=\left\{
  \begin{array}{ll}
  \medskip
  \sup\{u(z):z 	\in f^{-1}(x)\},  & f^{-1}(x)\neq \emptyset. \\
  \medskip

   0, & f^{-1}(x)= \emptyset.\\
  
   \end{array}
   \right.$$

Two useful results on Zadeh's extension are the following. The second one is an immediate consequence of the first one.

\begin{proposition}{\rm \cite{JSS}}\label{Prop:Zadeh's_level}
Let $X$ be a Hausdorff space.
If $f \colon X \to X$ is a continuous function, then $[\f(u)]_\alpha=f(u_\alpha)$ for each $u 	\in \F(X)$ and $\alpha 	\in 	\I	$.
\end{proposition}

\begin{proposition}
Let $X$ be a Hausdorff space. If $f \colon X \to X$ is a continuous function, then $\left(\f\right)^{n}=\widehat{f^{n}}$ for each $n 	\in \mathbb{N}$.
\end{proposition}

In the sequel, the latter result allows us to write  $\widehat{f^{n}}$ instead of $\left(\f\right)^{n}$. 

Given a non-empty subset $A$ of a Hausdorff space $X$, we denote by $\chi_A \colon X \to 	\I	$ the characteristic function of $A$. For the one-point set $\{x\}$, we put $\chi_x$ instead of $\chi_{\{x\}}$. If  $\K(X)$ denotes the hyperspace of all the non-empty compact subsets of $X$, we have the following proposition, which shows that $\f$ sends $\K(X)$ into 
itself.  

\begin{proposition}{\rm \cite{JSS}}
Let $f$ be a continuous function from a Hausdorff space $X$ into itself. Then $\f(\chi_K)=\chi_{f(K)}$ for each $K 	\in \K(X)$.
\end{proposition}

One of the central concepts of this paper is the notion of \emph{uniform space}. From our point of view, the most efficient approach to uniform spaces is via \emph{entourages of the diagonal}. We begin by introducing some notations. Let $X$ be a non-empty set, and let $A$ and $B$ be subsets of $X\times X$, i.e., relations on the set $X$. The inverse relation of $A$ will be denoted by $A\sp{-1}$ , and the composition of $A$ and $B$ will be denoted by $A\circ B$. Thus, we have that
$$A^{-1}=\{(x,y)\in X\times X: (y,x)\in A \}$$ and 
{\small $$
A\circ B=\{(x,y)\in X\times X: \text{ there exists }  z\in X \text{ such that } (x,z)\in A  \text{ and }  (z,y)\in B  \}.
$$}
The symbol $A^2$ stands for $A\circ A$ and $\Delta\sb{X}$ for the diagonal of $X$, that is, the subset  $\{(x,x):x\in X \}$ of $ X\times X$. Every set $A\sub X\times X$ that contains $\Delta\sb{X}$ is called an \emph{entourage of the diagonal}. We will denote by $\mathcal{D}_{X}$ the family of all entourages of the diagonal of $X$.

\begin{definition}{\rm
A \emph{uniformity} on a non-empty set $X$ is a subfamily  $\thinspace\U$ of $\thinspace\mathcal{D}_{X}$ that satisfies the following conditions: 

\begin{enumerate}[{\rm (U1)}] 
\item If $A\in \U $ and $A\sub B\in \mathcal{D}_X$, then $B\in \U.$

\item  If $A,B\in \U$, then  $U\cap V\in \U$.

\item  For every  $A\in \U$, there exists $B\in \U$  such that  $B^2\sub A$.

\item  For every  $A\in \U$, there exists $B\in \U$  such that  $B^{-1}\sub A$.

\item  $\bigcap_{A\in\U}A = \Delta_X$.
\end{enumerate}
}
\end{definition}

A \emph{uniform space} is a pair $(X,\U)$ consisting of a set $X$ and a uniformity $\U$ on the set $X$. Let $(X,\U)$ be a uniform space. A family $\mathcal{B}\sub \U$ is called a base for the uniformity $\U$ if for every  $A\in \U$, there exists $B\in \mathcal{B}$ such that $B\sub A$. The following result  is well known and easy to prove.

\begin{proposition} Let $X$  be a non-empty set. A non-empty family $\mathcal{B}$  of subsets of $X\times X$ is a base for some uniformity on $X$ if and only if it satisfies  the following properties:

\begin{enumerate}[{\rm(BS1)}]
\item For any $A,B\in \mathcal{B}$, there exists $C\in \mathcal{B}$ such that $C\subset A\cap B$.  

\item For every $A\in \mathcal{B}$, there exists  $B\in \mathcal{B}$   such that $B^{-1} \sub A$.

\item For every $A\in \mathcal{B}$, there exists  $B\in \mathcal{B}$   such that $B^{2} \sub A$.

\item  $\bigcap_{A\in \mathcal B}A=\Delta_{X}$.
\end{enumerate}
\end{proposition}

As usual,  a set $X$ equipped with a topology $\tau$ is called a \emph{topological space}, and it will be denoted   by $(X, \tau)$. It is a well-known fact that every uniformity $\U$ on a set $X$ induces a topology $\tau(\U)$ on $X$. To be precise, the topology $\tau(\U)$ is the family $\{V\sub X\, :\, \text{for every } x\in V \text{, there exists } U \in \U \text{ such that } U(x)\sub V\}$, where  $U(x)=\{y\in X\, : \, (x,y)\in U\}$. In this case, the topological space $(X,\tau(\U))$ is a Tychonoff space (for the details, we refer to the reader to Chapter~8 of the classic text  \cite{Engelking}).  For this reason, in the sequel, all the topological spaces are Tychonoff.
 
We turn to a brief discussion of the hyperspaces that we will consider in this paper. Given a Tychonoff topological space $(X,\tau)$, the symbols $\C(X)$ and $\K(X)$ denote, respectively,  the hyperspaces defined by 

$$
\begin{array}{l}
\C(X)=\{E\sub X \, : \, E \text{ is closed and non-empty}\},\medskip \\
\K(X)=\{E \in \C(X)\,: \, E \text{ is compact}\}.
\end{array}
$$
Thus, in the case of a uniform space $(X,\U)$,  $\C(X)$ (respectively, $\K(X)$) denotes the hyperspace of all non-empty closed (respectively, non-empty compact) subsets of $(X,\tau(\U))$. We will see that $\C(X)$ and $\K(X)$ can be endowed with a natural uniformity in this situation. 

Let $(X,\U)$ be a uniform space. For each $U\in \U$ and each $A\subset X$, let us  define  $U(A)=\bigcup_{x \in A}U(x)$.  Now, for each $U\in \U$ consider the families 
$$
\begin{array}{l}
\C[U]=\{(A,B)\in \C(X)\times \C(X):A\sub U(B),\,\, B \sub U(A)\}, \medskip\\
\K[U]=\{(A,B)\in \K(X)\times \K(X) \, : \, A\sub U(B),\,\, B \sub U(A)\}.

\end{array}
$$
Among the most interesting results in the theory of hyperspaces are the following three results. 

\begin{proposition} {\rm \cite{Bourbaki, Michael}} If $(X,\U)$ is a uniform space, then  $\{\K[U]: U\in \U\}$ is a base for a uniformity $\K(\U)$ on $\K(X)$. 
\end{proposition}

A remarkable result by Michael \cite{Michael} allows us to describe the topology induced by the uniformity $\K(\U)$. Let us recall that, for any topological space $(X,\tau)$, the topology $\tau$ induces a topology $\tau_V$ on $\C(X)$, the so-called \emph{Vietoris topology}, a base for $\tau_V$ is the family of all sets of the form 
$$
\mathcal{V}\left <V_1,\ldots, V_k\right>=\left\{ B\in \C(X)\, : \, B\subset \bigcup_{i=1}^{k}V_i \text{ and } B\cap V_i \not=\emptyset \text{ for } i=1,\ldots,k\right\},
$$
\noindent where $V_1,V_2,\ldots, V_n$ is a finite sequence of non-empty open sets of $X$.

\begin{theorem}{\rm \cite[Theorem 3.3]{Michael}}
If  $(X, \U)$ is a uniform space, then the topology induced by $\K(\U)$ on $\K(X)$ coincides with the Vietoris topology induced by $\tau(\U)$ on $\K(X)$. 
\end{theorem}

Allowing for the previous result, if no confusion can arise,   $\K(X)$ will denote the hyperspace of all non-empty compact subsets of $(X,\tau(\U))$ equipped  with the Vietoris topology induced by $\tau(\U)$.  For the hyperspace $\C(X)$, we have the following.

\begin{proposition} {\rm \cite{Bourbaki, Michael}} If $(X,\U)$ is a uniform space, then  $\{\C[U]: U\in \U\}$ is a base for a uniformity $\C(\U)$ on $\C(X)$. 
\end{proposition}

\cite[Example 2.10]{Michael} shows that the topology induced by the uniformity $\C(\U)$ on $\C(X)$ can be strictly finer than the Vietoris topology induced by $\tau(\U)$ on $\C(X)$. 

It is worth mentioning that a metric space $(X,d)$ has a compatible uniformity defined in a natural way. Indeed, it is routine to verify that the family $\U_d$, defined as the family of all sets 
$$V_\epsilon=\{(x,y)\in X\times X: d(x,y)<\epsilon\}$$
\noindent for each $\varepsilon>0$, is a base for a compatible uniformity on $(X,d)$. 
If $A,B\in \K(X)$, the \textit{Hasudorff distance} between $A$ and $B$ is defined as $d_H(A,B)= \mathrm{max}\{d(A,B), d(B,A) \}$, where $d(A,B)= \mathrm{sup} \{d(a,B):a\in A \}$ and $d(a,B)=\mathrm{inf}\{d(a,b):b\in B \}$. Notice that $d_H(A,B)$ is finite, since $A, B \in \K(X)$.

Given a metric space $(X,d)$, it is possible to define a metric $d_\infty$ on $\F(X)$ from the Hausdorff metric $d_H$ on $\K(X)$ as 
$$d_\infty (u,v)=\sup \{d_H (u_\alpha, v_\alpha): \alpha \in \I \}.$$
\noindent The metric $d_\infty$ is called the \textit{level-wise metric}. The symbol $\tau_\infty$ stands for the topology induced by $d_\infty$. 

Given a uniform space $(X,\U)$ and $U \in \U$, we define the set:
$$\F[U]=\{(u,v)\in \F(X)\times \F(X): (u_\alpha, v_\alpha)\in \K[U],\, \mathrm{~for~all~} \alpha \in \I \}.$$
\begin{proposition}{\rm \cite{JSS23}}
If $(X,\U)$ is a uniform space, then $\{\F[U]: U\in \U\}$ is a base for a uniformity  on  $\F(X)$.
\end{proposition}

We will denote by $\U_\infty$ the uniformity on $\F(X)$ defined in the previous proposition. It is called the \emph{level-wise uniformity}.

\begin{proposition}{\rm \cite{JSS23}}
If  $(X,d)$ is a metric space and $\U_d$ is the natural uniformity induced by $d$,  then the uniformity induced by $d_\infty$ on $\F(X)$ is the uniformity  $(\U_d)_\infty$.\end{proposition}

The following result plays an important role in the proof of Theorem \ref{Th:Conti-U_infinito}.

\begin{lemma}\label{Le:compact}
Let $K$ be a compact subset of a uniform space $(X, \U)$ and $f: X \to X$ a continuous function. Then for every $U \in \U$, there exists $V \in \U$ such that if $(x,y) \in V\cap (K \times X)$, then $(f(x),f(y)) \in U$.
\end{lemma}

\begin{proof}
Take $U \in \U$ and $U_1 \in \U$ such that $U^{2}_1 \sub U$. By the continuity of $f$, for every $x \in K$, there exists $W_x \in \U$ such that $f(W_x(x)) \sub U_1(f(x))$. For every $x \in K$, we choose a symmetric $V_x \in \U$ such that $V^{2}_x \sub W_x$.
 Since $K$ is compact, we can find $x_1, x_2,..., x_n \in K$ with $K \sub \bigcup^n_{i=1}V_{x_i}(x_i)$. Put $V=\bigcap^n_{i=1}V_{x_i}$. Take $x \in K$ and $y \in X$ such that $(x,y) \in V$. Then $x \in V_{x_i}(x_i) \sub W_{x_i}(x_i)$ for some $i=1,...,n$. So $(x_i,x) \in V_{x_i}$ and $(f(x), f(x_i)) \in U_1$. We also have that $(x,y) \in V \sub V_{x_i}$.  So $(x_i, y) \in V^{2}_{x_i} \sub W_{x_i}$. Hence, $y \in W_{x_i}(x_i)$. So $(f(x_i),f(y)) \in U_1$. We conclude that $(f(x), f(y)) \in U^{2}_1 \sub U$.
\end{proof}

\begin{theorem}\label{Th:Conti-U_infinito}
Let $(X, \U)$ be a uniform space. If $f: (X, \U) \to (X, \U)$ is continuous, then $\f: (\F(X), \U_{\infty}) \to (\F(X), \U_{\infty})$ is continuous.
\end{theorem}

\begin{proof}
Suppose that $f: (X, \U) \to (X, \U)$ is continuous. Take $u \in \F(X)$ and $U \in \U$. Pick a symmetric $W \in \U$ such that $W \sub U$. By Lemma \ref{Le:compact}, there exists $V \in \U$ such that $x \in u_0=K$, $y \in X$, and $(x,y) \in V$ imply $(f(x), f(y)) \in W$. Let us show that $\f(\F[V](u)) \sub \F[U](\f(u))$. Indeed, if $v \in \F[V](u)$, then $(u_\alpha, v_\alpha)\in \K[V]$ for each $\alpha \in [0,1]$, i.e., $u_\alpha \sub V(v_\alpha)$ and $v_\alpha \sub V(u_\alpha)$ for each $\alpha \in [0,1]$. Fix $\alpha \in [0,1]$. If $x \in u_\alpha \sub K$, there exists $y \in v_{\alpha}$ such that $(x,y) \in V$. By the choice of $V$, we can conclude that $(f(x), f(y)) \in W$. So $f(x) \in W(f(y)) \sub W(f(v_\alpha))$. Hence $f(u_\alpha) \sub W(f(v_\alpha))$. Now if $y \in u_\alpha$, there exists $x \in u_{\alpha} \sub K$ such that $(x,y) \in V$. By the choice of $V$, we can conclude that $(f(x), f(y)) \in W$. So $f(v_\alpha) \sub W(f(u_\alpha))$. It follows that $(f(u_\alpha), f(v_\alpha)) \in \K[W]$. By Proposition \ref{Prop:Zadeh's_level}, the latter fact is equivalent to $([\f(u)]_\alpha, [\f(v)]_\alpha) \in \K[W]$. We have thus proved that $\f(v) \in \F[U](\f(u))$. This completes the proof.
\end{proof}

We now turn our attention to the Skorokhod uniformity. 

The symbol $\mathbb{T}$ denotes the family of all increasing homeomorphisms from $[0,1]$ onto itself. Given $t \in \mathbb{T}$, we define $\Vert t\Vert=\sup\{|t(\alpha)-\alpha|: \alpha \in \I\}$. 

Let $(X,d)$ be a metric space. The \textit{Skorokhod metric} on $\F(X)$ is defined as  
{\small $$d_0(u,v)=\inf\{\epsilon: \text{there exists }  t \in \mathbb{T}\, \text{ such that }\, \Vert t\Vert \leq \epsilon\, \text{ and }\, d_{\infty}(u,tv)\leq \epsilon\},$$}
\noindent for all $u,v\in \F(X)$.

Let $(X,\U)$ be a uniform space. If $U \in \U$ and $\epsilon >0$, we consider the set 
\begin{flushleft}
$G[U,\epsilon ]=\{(u,v)\in \F(X)\times \F(X):$ 
\end{flushleft}
\begin{flushright}
$\mathrm{exists~} t\in \mathbb{T} \mathrm{~with~} \Vert t\Vert <\epsilon  \mathrm{~ and ~} (u_\alpha, v_{t(\alpha )})\in \K[U] \mathrm{~for~all~} \alpha \in \I \}$.
\end{flushright}

\begin{proposition}{\rm \cite{JSS23}}
If $(X,\U)$ is a uniform space, then $\{G[U,\epsilon]: U\in \U \mathrm{~and~} \epsilon >0\}$ is a base for a uniformity  on  $\F(X)$.
\end{proposition}

The uniformity on $\F(X)$ defined in the previous proposition will be denoted by $\mathcal{U}_0$, and it is called the \emph{Skorokhod uniformity}. We have the following relation between the Skorokhod metric and the Skorokhod uniformity.

\begin{proposition}{\rm \cite{JSS23}}
If $(X,d)$ is a metric space, then the uniformity $\U_{d_0}$ on $\F(X)$ coincides with the uniformity $(\U_d)_0$. \end{proposition}

The following result plays an important role in the proof of Theorem \ref{Th:Conti-U_0}.

\begin{theorem}\label{Th:linearity}{\rm \cite{JSS}}
If $X$ is a Hausdorff space and $f:X \to X$ is continuous, then $\f(tu)=t \f(u)$ for each $u \in \F(X)$ and $t \in \mathbb{T}$.
\end{theorem}

\begin{theorem}\label{Th:Conti-U_0}
Let $(X, \U)$ be a uniform space. If $f: (X, \U) \to (X, \U)$ is continuous, then $\f: (\F(X), \U_0) \to (\F(X), \U_0)$ is continuous.
\end{theorem}

\begin{proof}
Suppose that $f: (X, \U) \to (X, \U)$ is continuous. Take $U \in \U$, $\epsilon>0$, and $u \in \F(X)$. By Theorem \ref{Th:Conti-U_infinito}, there exists $V \in \U$ such that $\f(\F[V](u)) \sub \F[U](\f(u))$. Pick $v \in G[V,\epsilon ](u)$. Then there exists $t \in \mathbb{T}$ such that $\Vert t\Vert <\epsilon$ and $(u_\alpha, v_{t(\alpha )})\in \K[V]$ for all $\alpha \in [0,1]$. Hence $t^{-1}v \in \F[V](u)$. By the choice of $V$, we have that $(\f(u),\f(t^{-1}v)) \in \F[U]$. It follows from Theorem \ref{Th:linearity} that $(\f(u), t^{-1}\f(v)) \in \F[U]$. So $([\f(u)]_\alpha, [\f(v)]_{t(\alpha)}) \in \K[U]$ for each $\alpha \in [0,1]$. Since $\Vert t\Vert <\epsilon$, we conclude that $(\f(u), \f(v)) \in G[U, \epsilon]$. We have thus proved that $\f(G[V, \epsilon](u)) \sub G[U, \epsilon](\f(u))$. This completes the proof.
\end{proof}

Let $(X,\tau)$ be a topological space. If $u \in \F(X)$, then the \textit{endograph} of $u$ is defined as $end(u)=\{(x,\alpha )\in X\times \I: u(x)\geq \alpha\}$. Notice that $end (u)\in \C(X\times \I)$. We also define the \textit{sendograph} of $u$ by $send(u)= end(u)\cap (u_0\times \I) $. Observe that $send (u)\in \K(X\times \I)$.

Consider the uniformity $\U_{\I} $  defined on $\I$ by means of the base $\{V_\epsilon:\epsilon >0 \}$, where $ V_\epsilon =\{(\alpha ,\beta )\in \I \times \I :\vert \alpha -\beta \vert <\epsilon \} $. Let $(X, \U)$ be a uniform space. Given $U \in \U$ and $\epsilon>0$, we define the following sets:
$$S[U, \epsilon]=\{(u,v) \in \F(X) \times \F(X): (send(u), send(v))\in \K[U \times V_\epsilon]\}.$$
The family $\{S[U, \epsilon]: U \in \U,\, \epsilon>0\}$ is base for a uniformity
$\mathcal{U}_S$ on $\F(X)$. The uniformity $\mathcal{U}_S$ is called the \emph{sendograph uniformity} \cite{JSS23}.

Consider now a metric space $(X,d)$. Define the metric $\overline{d}$ on $X\times \I$ as follows: 
$$\overline{d}((x,a),(y,b))=\max\{d(x,y),\vert a-b\vert\}.$$
The \textit{sendograph metric} $d_S$ on $\F(X)$ is the Hausdorff metric $\overline{d}_H$ (on $\K(X\times \I)$) between the non-empty compact subsets $send(u)$ and $send(v)$ for every $u, v \in \F(X)$ (see \cite{Kupka2011}).

\begin{proposition}{\rm \cite{JSS23}} Let $(X,d)$ be a metric space. Then $\U_{d_S}=(\U_d)_S$ and $\U_{d_E}=(\U_d)_E$.
\end{proposition}

Consider a topological space $X$ and a continuous function $f: X \to X$. We define the extension of $f$ to $\K(X)$ as the function $\overline{f}: \K(X) \to \K(X)$ with $\overline{f}(A)=f(A)$ for each $A \in \K(X)$.

The following result is easy to show. So we omit its proof.

\begin{lemma}\label{Le:send}
If $X$ is a Hausdorff space, $u \in \F(X)$ and $f:X \to X$ is continuous, then $send(\f(u))=\overline{f_I}(send(u))$, where $f_I=f\times Id_{[0,1]}$.
\end{lemma}

We now show the continuity of Zadeh's extension with respect to the sendograph uniformity.

\begin{theorem}\label{Th:Conti-U_S}
Let $(X, \U)$ be a uniform space. If $f: (X, \U) \to (X, \U)$ is continuous, then $\f: (\F(X), \U_S) \to (\F(X), \U_S)$ is continuous.
\end{theorem}

\begin{proof}
Suppose that $f: (X, \U) \to (X, \U)$ is continuous. Take $U \in \U$, $\epsilon>0$, and $u \in \F(X)$.  Since $\overline{f_I}: \K(X\times \I) \to \K(X\times \I)$ is continuous, we can find $W \in \U$ and $\delta>0$ such that $\overline{f_I}(\K[W \times V_\delta](send(u))) \sub \K[U \times V_\epsilon](\overline{f_I}(send(u)))$. Pick $v \in S[W,\delta ](u)$. Then
$send(v)\in \K[W \times V_\delta](send(u))$.
By Lemma \ref{Le:send}, we have that $send(\f(v)) \in \K[U \times V_\epsilon](send(\f(u)))$, i.e., $\f(v) \in S[U, \epsilon](\f(u))$. We conclude that $\f(S[W, \delta](u)) \sub S[U, \epsilon](\f(u))$. This proves the continuity of $\f: (\F(X), \U_S) \to (\F(X), \U_S)$.
\end{proof}

We have the following relationship between the three uniformities defined above.

\begin{proposition}{\rm \cite{JSS23}}\label{Prop:Topologies-incluisions}
Consider a uniform space $(X, \U)$. Then $\tau(\U_S) \sub \tau(\U_0) \sub \tau(\U_{\infty})$.
\end{proposition}

\section{Transitivity on some uniformities}

A \textit{dynamical system} is a pair $(X,f)$, where $X$ is a topological space and $f \colon X \to X$ is a continuous function. 
In this section, we mainly characterize the transitivity of the dynamical systems $\f \colon (\F(X), \U_\infty) \to (\F(X), \U_\infty)$,  $\f \colon (\F(X), \U_0) \to (\F(X), \U_0)$, and $\f \colon (\F(X), \U_S) \to (\F(X), \U_S)$ (see Theorem \ref{Th:transitivity}). 

Let $X$ be a topological space and $f \colon X \to X$ be a continuous function. 
Let us recall that a dynamical system $(X,f)$ is \textit{transitive} if for every non-empty open subsets $U$ and $V$ of $X$, there exists $n 	\in \mathbb{N}$ such that $f^{n}(U) \cap V \neq \emptyset$. We say that $(X,f)$ is \textit{weakly mixing} if $f\times f \colon X\times X \to X \times X$ is transitive. Let us recall that $\overline{f}\colon \K(X) \to \K(X)$ is defined by $\overline{f}(K)=f(K)$ for each $K 	\in \K(X)$.

Let $f \colon X \to X$ be a continuous function on a topological space $X$. Banks \cite{Banks} and Peris \cite{Peris} showed that $(X,f)$ is weakly mixing if and only if $(\mathcal{K}(X),\overline{f})$ is transitive. To be precise, they show the following.


\begin{theorem}\label{Th:caso_autonomo}{\rm \cite{Banks, Peris}}
Let $f \colon X \to X$ be a continuous function on a topological space $X$. Then the following conditions are equivalent:

\begin{itemize}
	\item[{\rm(1)}] $(X,f)$ is weakly mixing.

	\item[{\rm(2)}] $(\mathcal{K}(X),\overline{f})$ is weakly mixing.

	\item[{\rm(3)}] $(\mathcal{K}(X),\overline{f})$ is transitive.
\end{itemize}
\end{theorem}

A dynamical system $(X,f)$ is \textit{weakly mixing of order $m$} ($m \geq 2$) if the function
$$\underbrace{f \times \cdots \times f}_{m-times} \colon X^{m} \to X^{m}$$
\noindent is transitive. We need the following five results in the proof of Theorem \ref{Th:transitivity}.

\begin{theorem}{\rm \cite[Theorem 1]{Banks}}\label{Th:Banks}
If $(X,f)$ a is weakly mixing dynamical system, then $(X,f)$ is weakly mixing of all orders.
\end{theorem}

\begin{proposition}{\rm \cite{JSS} }\label{Prop_limits}
Let $X$ be a Hausdorff space and $u	\in \mathcal{F}(X)$. If $L_{u} \colon \mathbb{I} \to (\mathcal{K}(X),\tau_V)$ is the function defined by $L_{u}(\alpha )= u_\alpha$ for all $\alpha	\in I$, then the following holds:

\begin{itemize}
	\item[{\rm i)}] $L_u$ is left continuous on $(0,1]$;

	\item[{\rm ii)}]  $\lim\limits_{\lambda \rightarrow \alpha ^+}L_u(\lambda)=\overline{\bigcup_{\beta>\alpha}u_\beta}$ and $\lim\limits_{\lambda \rightarrow \alpha ^+}L_u(\lambda)\subset u_\alpha$ for each $\alpha 	\in [0,1)$;

	\item[{\rm iii)}] $L_u$ is right continuous at $0$.
\end{itemize}

\noindent Conversely, for any decreasing family $\{u_\alpha :\alpha 	\in 	\I	 \}\sub \mathcal{K}(X)$ satisfying i)--iii), there exists a unique  $w	\in \mathcal{F}(X)$ such that $w_\alpha =u_\alpha $ for every $\alpha 	\in 	\I	$.
\end{proposition}

Let  $X$ be a Hausdorff space. For any  $u	\in \mathcal{F}(X)$ and $\alpha 	\in [0,1)$, define $u_{\alpha ^+}=\lim\limits_{\lambda \rightarrow \alpha ^+}L_u(\lambda)$. It follows from ii) of the previous proposition that  $L_u$ is right continuous at $\alpha $ if and only if $u_{\alpha ^+}= u_\alpha$.

\begin{proposition}\label{Prop:monotonic}{\rm \cite{JSS23}}
Let $(X, \U)$ be a uniform space. If $W \in \U$, the sets $A,B,C,F,G, H$ are in $\mathcal{K}(X)$ and $A\sub B\sub C$, then we have the following:

\begin{itemize}
	\item[{\rm i)}] If $(A,C) \in \K[W]$, then $(A,B) \in \K[W]$.

	\item[{\rm ii)}] If $(A,C) \in \K[W]$, then $(B,C) \in \K[W]$.
	
	\item[{\rm iii)}] If $(A,F) \in \K[W]$ and $(G, H) \in \K[W]$, then $(A \cup G, F \cup H) \in \K[W]$.
\end{itemize}
\end{proposition}

\begin{lemma}\label{Le:partition}{\rm \cite{JSS23}}
Suppose that $(X,\U)$ is a uniform space. If $u	\in \mathcal{F}(X)$, then for each $W \in \U$, there exist numbers $0=\alpha _0 <\alpha _1 < \dots < \alpha _n =1 $ such that $(u_{\alpha^{+}_{k}} ,u_{\alpha_{k+1}}) \in \K[W]$ for $k=0,1,\dots ,n-1$.
\end{lemma}

The following result is an immediate consequence of Proposition \ref{Prop:monotonic}.

\begin{lemma}\label{Le:refinament}
Suppose that $(X,\U)$ is a uniform space. Take $u	\in \mathcal{F}(X)$, $W \in \U$, and a partition $0=\alpha _0 <\alpha _1 < \dots < \alpha _n =1 $ such that $(u_{\alpha^{+}_{k}} ,u_{\alpha_{k+1}}) \in \K[W]$ for $k=0,1,\dots ,n-1$. If the partition $0=\beta_0<\beta_1<\cdots<\beta_m=1$ is a refinement of the partition $0=\alpha _0 <\alpha _1 < \dots < \alpha _n =1 $, then $(u_{\beta^{+}_{k}} ,u_{\beta_{k+1}})\in \K[W]$ for $k=0,1,\dots ,m-1$.
\end{lemma}

We are now ready to present the main result of this paper.

\begin{theorem}\label{Th:transitivity}
Let $(X,\U)$ be a uniform space. Then the following conditions are equivalent:

\begin{itemize}
	\item[{\rm i)}] $f: (X, \U) \to (X, \U)$ is weakly mixing;

	\item[{\rm ii)}] $\overline{f}:(\mathcal{K}(X),\K(\U)) \to (\mathcal{K}(X),\K(\U))$ is transitive;

	\item[{\rm iii)}] $\f:(\F(X), \U_{\infty}) \to (\F(X), \U_{\infty}) $ is transitive;

	\item[{\rm iv)}] $\f:(\F(X), \U_{0}) \to (\F(X), \U_{0}) $ is transitive;

	\item[{\rm v)}] $\f:(\F(X), \U_{S}) \to (\F(X), \U_{S}) $ is transitive.
\end{itemize}

\end{theorem}

\begin{proof}
By Theorem \ref{Th:caso_autonomo}, we have that i) $\Leftrightarrow$ ii). We have that iii) $\Rightarrow$ iv) $\Rightarrow$ v) by Proposition \ref{Prop:Topologies-incluisions}.

Let us show that ii) implies iii). Suppose that $\overline{f}:(\mathcal{K}(X),\K(\U)) \to (\mathcal{K}(X),\K(\U))$ is transitive. Let us show that $\f:(\F(X), \U_{\infty}) \to (\F(X), \U_{\infty}) $ is transitive.
Take $u,v 	\in \F(X)$ and $U_1, V_1 \in \U$. Put $P=\F[U_1](u)$ and $Q=\F[V_1](v)$.
Take $U,V \in \U$ such that $U^{2} \sub U_1$ and $V^{3} \sub V_1$.
By Lemma \ref{Le:partition}, there exist numbers $0=\alpha _0 <\alpha _1 < \dots < \alpha _n =1 $ such that $(u_{\alpha^{+}_{k}} ,u_{\alpha_{k+1}})\in \K[U]$ for each $k=0,1,\dots ,n-1$. Also, there exist numbers $0=\beta_0 <\beta_1 < \dots < \beta_m =1 $ such that $(v_{\beta^{+}_{k}} ,v_{\beta_{k+1}})\in \K[V]$ for all $k=0,1,\dots ,m-1$. By Lemma \ref{Le:refinament}, we can assume that $n=m$ and $\alpha_i=\beta_i$ for each $i=0,1,...,n$. We shall show that for every $k=0,1,\dots ,n-1$, we have the following:

\begin{equation}\label{Eq:5}
(u_\beta, u_{\alpha_{k+1}}) \in \K[U],\,\,\,\, \textit{if}\,\, \beta	\in (u_{\alpha_{k}} ,u_{\alpha_{k+1}}].
\end{equation}

\noindent For this, notice that $u_{\alpha_{k+1}}\sub u_\beta \sub u_{\alpha^{+}_{k}}$ and $(u_{\alpha_{k+1}}, u_{\alpha^{+}_{k}}) \in \K[U]$.
Proposition \ref{Prop:monotonic} implies that $(u_{\alpha_{k+1}}, u_\beta)\in \K[U]$. This shows (\ref{Eq:5}). 

Since $\overline{f}:(\mathcal{K}(X),\K(\U)) \to (\mathcal{K}(X),\K(\U))$ is transitive, Theorem \ref{Th:caso_autonomo} implies that $(\mathcal{K}(X),\overline{f})$ is weakly mixing. So Theorem \ref{Th:Banks} tells us that the dynamical system  $(\mathcal{K}(X),\overline{f})$ is weakly mixing of all orders. Therefore, there exist $m>0$ and $K_1, K_2,..., K_n, L_1, L_2,..., L_n 	\in \mathcal{K}(X)$ such that for each $1 \leq i \leq n$ we have the following:

\begin{equation}\label{Eq:1}
(u_{\alpha_i}, K_i)\in \K[U],
\end{equation}

\begin{equation}\label{Eq:2}
(v_{\alpha_i}, L_i)\in \K[V],
\end{equation}

\begin{equation}\label{Eq:3}
(\overline{f}^{m}(K_i), L_i)\in \K[W]. 
\end{equation}

\noindent Put $w_{\alpha_i}=\bigcup_{k\geq i} K_k$ for each $1\leq i \leq n$. Proposition \ref{Prop:monotonic} and (\ref{Eq:1}) imply that for every $1\leq i \leq n$, we have 

\begin{equation}\label{Eq:4}
(w_{\alpha_i}, u_{\alpha_i})\in \K[V]. 
\end{equation}

\noindent Let us define $w_{\alpha}$ for each $\alpha 	\in 	\I	$ as follows:

$$w_\alpha=\left\{
  \begin{array}{ll}
  \medskip
  w_{\alpha_1},  & 0\leq \alpha \leq \alpha_1. \\
  \medskip

   w_{\alpha_i}, & \alpha 	\in (\alpha_{i-1},\alpha_i]\,\, \& \,\, i=2,...,n.\\
  
   \end{array}
   \right.$$
   
\noindent The family $\{w_{\alpha}:\alpha	\in 	\I	\}$ satisfies the conditions of Proposition \ref{Prop_limits}. Hence, it determines an element $w 	\in \F(X)$.
Let us show that $w \in	\F[U_1](u)$. Take $\alpha 	\in 	\I	$. Suppose that $\alpha 	\in [0, \alpha_1]$. Then $w_\alpha=w_{\alpha_1}$. Proposition \ref{Prop:monotonic} and the choice of $\alpha_0=0$ and $\alpha_1$ imply that $(u_\alpha, u_{\alpha_1})\in \K[U]$. The latter fact and (\ref{Eq:4}) imply that $(w_\alpha, u_\alpha) \in \K[U]^{2}\sub \K[U^2] \sub \K[U_1]$ for each $\alpha 	\in [0, \alpha_1]$.
We now take $\alpha 	\in (\alpha_{i-1}, \alpha_i]$ for some $1<i\leq n$. Facts (\ref{Eq:5}) and (\ref{Eq:4}) imply that $(w_\alpha, u_\alpha) \in \K[U]^{2} \sub \K[U^2] \sub \K[U_1]$ for each $\alpha \in (\alpha_1,1]$. We can conclude that $(w_\alpha, u_\alpha) \in \K[U_1]$ for each $\alpha \in [0,1]$. Then $(u,w) \in \F[U_1]$. Hence, $w 	\in \F[U_1](u)=P$.

Put $z_{\alpha_i}=\bigcup_{k\geq i} L_k$ for each $1\leq i \leq n$. Let us define $z_{\alpha}$ for each $\alpha 	\in 	\I	$ as 

$$z_\alpha=\left\{
  \begin{array}{ll}
  \medskip
  z_{\alpha _1},  & 0\leq \alpha \leq \alpha_1, \\
  \medskip

   z_{\alpha_i}, & \alpha 	\in (\alpha_{i-1},\alpha_i]\,\, \& \,\, i=2,...,n.\\
  
   \end{array}
   \right.$$

\noindent The family $\{z_{\alpha }:\alpha	\in 	\I	\}$ satisfies the conditions of Proposition \ref{Prop_limits}. Using (\ref{Eq:2}), we can argue as in $w$ to prove the following: 

\begin{equation}\label{Eq10}
(v,z) \in \F[V]^{2}.
\end{equation}

By Proposition \ref{Prop:monotonic} and (\ref{Eq:3}), we have that $(\overline{f}^{m}(w_{\alpha_i}), z_{\alpha_i})\in \K[V]$ for every $i=1,2,...,n$. Since $\overline{f}^{m}(w_{\alpha_i})=f^{m}(w_{\alpha_i})=[\widehat{f^{m}}(w)]_{\alpha_i}$, we conclude that
$([\widehat{f^{m}}(w)]_{\alpha_i}, z_{\alpha_i}) \in \K[V]$ for each $i=1,2,...,n$. Definitions of $w$ and $z$ imply that $([\widehat{f^{m}}(w)]_{\alpha}, z_{\alpha}) \in \K[V]$ for every $\alpha 	\in 	\I	$. Therefore, $(\widehat{f^{m}}(w), z) \in \F[V]$. The latter fact and (\ref{Eq10}) imply that
$(\widehat{f^{m}}(w), v)\in \F[V]^{3} \sub \F[V_1]$. It follows that $\widehat{f^{m}}(w)	\in \F[V_1](v)=Q$. We have thus proved that $\widehat{f^{m}}(w) 	\in \widehat{f^{m}}(P) \cap Q \neq \emptyset$.
Therefore, $\f: (X, \U_{\infty}) \to (X, \U_{\infty})$ is transitive.

\medskip

Finally, let us  prove that v) implies ii). Suppose that $\f:(\F(X), \U_{S}) \to (\F(X), \U_{S}) $ is transitive. Take $K,L 	\in \K(X)$ and $U, V\in \U$. Define $u=\chi _K$ and $v=\chi _L$, which clearly are elements of  $\F(X)$.  From the transitivity of $\f:(\F(X), \U_{S}) \to (\F(X), \U_{S})$,  it follows the existence of $w	\in S[U,1](u)$ and $n	\in \mathbb{N}$ such that $\f^n(w) 	\in S[V,1](v)$. Put $A=w_0$. Then $(send(w), send(u))\in \K[U\times V_1]$ and $(send(\widehat{f^n}(w)), send(v))\in \K[V\times V_1]$. So $A\times\{0\} \sub send(w) \sub [U\times V_1](send(u))$. Hence, for each $a \in A$, we can find $(x_a,\alpha) \in send(u)$ such that $(a, x_a) \in U$, whence $a \in U(x_a) \sub U(K)$. It follows that $A \sub U(K)$. On the other hand, $K\times\{0\} \sub send(u) \sub [U\times V_1](send(w))$. Hence, for each $k \in K$, we can find $(x_k,\alpha) \in send(w)$ such that $(k, x_k) \in U$, whence $k \in U(x_k) \sub U(A)$. It follows that $K \sub U(A)$. Therefore, $A \in \K[U](K)$. Similarly, we can show that $\overline{f}^{n}(A) \in \K[V](L)$. This completes the proof.
\end{proof}


Daniel Jard\'on: Academia de Matem\'aticas, Universidad Aut\'onoma de la Ciudad de M\'exico, Calz. Ermita Iztapalapa S/N, Col. Lomas de Zaragoza 09620, M\'exico D.F., Mexico. E-mail: daniel.jardon@uacm.edu.mx. ORCID: 0000-0001-9054-0788

\medskip

Iv\'an S\'anchez: Departamento de Matem\'aticas, Universidad Aut\'onoma Metropolitana, Av. San Rafael Atlixco 186, Col. Vicentina, Del. Iztapalapa, C.P. 09340, Mexico city, Mexico. E-mail: isr@xanum.uam.mx. ORCID: 0000-0002-4356-9147

\medskip

Manuel Sanchis: Institut de Matem\`{a}tiques i Aplicacions de Castell\'o (I\-M\-AC), Universitat Jaume I, Spain. E-mail: sanchis@mat.uji.es. ORCID: 0000-0002-5496-0977


\begin{thebibliography}{100}

\bibitem{Banks} J. Banks, Chaos for induced hyperspace maps, Chaos Solitons and Fractals 25 (2005) 1581--1583. https://doi.org/10.1016/j.chaos.2004.11.089.


\bibitem{Bourbaki} N. Bourbaki, General Topology, Springer Berlin, Heidelberg, 1995. https://doi.org/10.1007/978-3-642-61701-0.


\bibitem{Engelking} R. Engelking, General Topology, Sigma Series in Pure Mathematics, vol. 6, Heldermann Verlag, Berlin, 1989. ISBN-13: 978-3885380061.


\bibitem{JSS23} D. Jard\'on, I. S\'anchez and M. Sanchis,
\newblock Fuzzy sets on uniform spaces, Iranian Journal of Fuzzy Systems 20 (6) (2023) 123--135. https://doi.org/10.22111/IJFS.2023.43774.7700.

\bibitem{JSS} D. Jard\'on, I. S\'anchez and M. Sanchis,
\newblock Some questions about Zadeh's extension on metric spaces, Fuzzy Sets and Systems 379 (2020) 115--124. https://doi.org/10.1016/j.fss.2018.10.019.

\bibitem{Transitivity} D. Jard\'on, I. S\'anchez and M. Sanchis,
\newblock Transitivity in fuzzy hyperspaces, Mathematics 8 (11): 1862 (2020).\\ https://doi.org/10.3390/math8111862.


\bibitem{Kupka2011} J. Kupka,
\newblock On fuzzifications of discrete dynamical systems,
\newblock Information Sciences 181 (2011) 2858--2872.\\ https://doi.org/10.1016/j.ins.2011.02.024.


\bibitem{Michael} E. Michael,
\newblock Topologies on spaces of subsets, Transactions of the American Mathematical Society, Vol. 71, No. 1 (1951) 152--182. DOI:10.1090/S0002-9947-1951-0042109-4.


\bibitem{Peris} A. Peris, Set-valued discrete chaos, Chaos, Solitons and Fractals 26 (2005) 19--23. https://doi.org/10.1016/j.chaos.2004.12.039.


\bibitem{RF} H. Rom\'an-Flores, A note on in set-valued discrete systems, Chaos, Solitons and Fractals 17 (2003) 99–-104.\\ https://doi.org/10.1016/S0960-0779(02)00406-X.














\end{thebibliography}
\end{document}